\documentclass{article}
\usepackage{graphicx} 
\usepackage{amsmath}
\usepackage{amsthm}
\usepackage{amssymb}
\usepackage{hyperref} 
\usepackage[a4paper, total={6in, 8in}]{geometry}
\newtheorem{theorem}{Theorem}[section]
\newtheorem{corollary}[theorem]{Corollary}
\newtheorem{lemma}[theorem]{Lemma}
\newtheorem{proposition}[theorem]{Proposition}
\newtheorem{definition}[theorem]{Definition}
\newtheorem{example}[theorem]{Example}
\newtheorem{remark}[theorem]{Remark}
\title{The oriented swap process on half line}
\author{Yuan Tian}
\date{\today}

\begin{document}

\maketitle

\begin{abstract}
    In this paper we study the oriented swap process on the positive integers and its asymptotic properties. Our results extend a theorem by Angel, Holroyd, and Romik regarding the trajectories of particles in the finite oriented swap process. Furthermore, we study the evolution of the type of a particle at the leftmost position over time. Our approach relies on a relationship between multi-species particle systems and Hecke algebras, complemented by a detailed asymptotic analysis. 
\end{abstract}

\section{Introduction}
The oriented swap process, one of the two main models in the study of random sorting networks, was first explored by Angel, Holroyd, and Romik \cite{angel2009oriented}. Unlike the uniform random sorting networks, which were initially studied in \cite{angel2007random}, the oriented swap process can be naturally interpreted as a multi-colored interacting particle system with states in a permutation group. Specifically, a finite oriented swap process with $n$ particles is a continuous-time Markov process that takes values in the permutation group $\mathcal{S}_n$, denoted by $\pi_{t}^{n}$. For $i \in {1, 2, \dots n-1}$, let $\tau_{i} \in \mathcal{S}_n$ be the transposition of $i$ and $i + 1$. The transition rate from a configuration $\sigma$ to $\tau_{i}\sigma$ is $1$ if $\sigma(i) < \sigma(i + 1)$. \footnote{In this paper, $\tau \sigma(i) = \tau(\sigma(i))$. Under this convention, the left multiplication by a transposition swaps the positions of particles with label $1$ and $2$, while the right multiplication swaps the particles at positions $1$ and $2$.}

Angel, Holroyd, and Romik \cite{angel2009oriented} established several asymptotic properties of the finite oriented swap process with finite $n$, including the convergence of the trajectories, the fluctuations of the finishing time, and the convergence of the number of inversions in the permutations. Subsequent studies, such as \cite{bisi2022oriented}, \cite{bufetov2022absorbing}, \cite{zhang2023shift}, have further explored the fluctuations of the absorbing time utilizing various symmetries of the multi species-TASEP. 

However, these studies primarily focus on properties for a finite $n$.  We extend the analysis to the oriented swap process on all positive integers in this paper. We name it as the oriented swap process on the half line, whose configuration at time $t \in \mathbb{R}_{>0}$ is denoted as $\pi_t$. 

\begin{definition}
    The swap process on the half line $\pi_{t}$ is a stochastic process that takes values in a measurable space $(\Omega, \mathcal{F})$. $\Omega$ is the set of all permutations on $\mathbb{Z}^+$, i.e. for a fixed $t$ $\pi_{t}, \mathbb{Z}^+ \rightarrow \mathbb{Z}^+$ is a random bijection of $\mathbb{Z}^+$. The sigma algebra $\mathcal{F}$ is generated by the union of all the subsets of finite random permutations $\mathcal{S}_n$ for all $n$. Its value at time $t$ is defined as the union of finite swap process on non-intersect segments(see Harris construction in Section \ref{sec:Pre}). In our notation, $\pi_{t}(l)$ is the position of the particle $l$ at time $t$, and $\pi_{t}^{-1}(l)$ is the label of the particle at position $l$ at time $t$. Later in this paper, we sometimes abuse the notation $\pi_{t}$ as the configuration of the infinite swap process. 
\end{definition}

We also have the inverse of this permutation, denoted as $\pi_t^{-1}$. Specifically, $\pi_t(i) = j$ indicates that, at time $t$, particle $i$ is at position $j$, while $\pi_t^{-1}(i) = j$ indicates that, at time $t$, the particle at position $i$ is the one with the label $j$, see Figure \ref{fig:1} and Figure \ref{fig:2} as examples. 

\begin{example}
    \begin{figure}
    \centering
    \includegraphics[scale = 0.4]{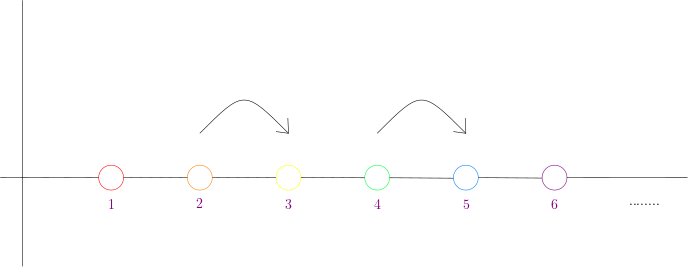}
    \caption{The oriented swap process at time $0$, $\pi_t(i) = i, \forall i$}
    \label{fig:1}
    \end{figure}

    \begin{figure}
    \centering
    \includegraphics[scale = 0.4]{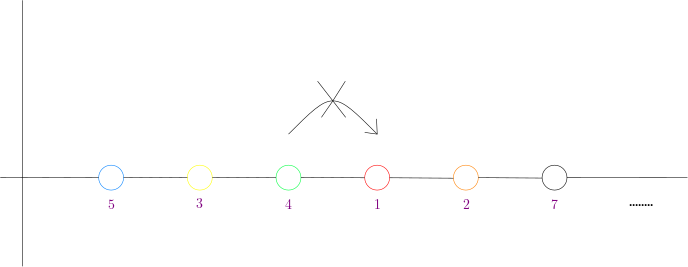}
    \caption{The oriented swap process after several jumps}
    \label{fig:2}
    \end{figure}
    In Figure \ref{fig:1} and Figure \ref{fig:2}, an arrow is a swap attempt, a cross on the arrow shows the swap is suppressed since the particle on the left is weaker. In Figure \ref{fig:2}, $(\pi_t(1), \pi_t(2), \pi_t(3), \pi_t(4), \pi_t(5)) = (4, 5, 2, 3, 1) $ and  $(\pi_t^{-1}(1), \pi_t^{-1}(2), \pi_t^{-1}(3), \pi_t^{-1}(4), \pi_t^{-1}(5), \pi_t^{-1}(6)) = (5,3,4,1,2,7) $. 
\end{example}

While some similar questions can be posed as in the finite case, such as the convergence of trajectories, other properties from the finite model no longer apply in the infinite setting. However, the half-line model introduces new and intriguing questions. We are particularly interested in the following one. Which particle occupies the position $1$ at time $t$? In the finite setting, the leftmost position will eventually be occupied by the weakest particle $n$. However, in the infinite system, there will always be weaker particles that get pushed back to this position. This naturally leads to the question of the asymptotic distribution of $\pi_{t}^{-1}(1)$ for large $t$. To the best of our knowledge, no purely probabilistic tools are available to describe which particle is at position 1 at time t. However, we can employ algebraic tools to transform the problem into a study of $\pi_{t}(1)$. After applying the symmetry property, it is somewhat unexpected to observe that $\pi_{t}^{-1}(1)$ follows a Poisson distribution. Moreover, we establish the law of large numbers and the Gaussian fluctuation of $\pi_{t}^{-1}(1)$. More generally, the first main result of this paper is as follows.

\begin{theorem}[Theorem \ref{thm3}, Corollary \ref{thm4}]\label{thmtyy}
    Let $\alpha ,s \in \mathbb{R}$, one has
    \begin{enumerate}
        \item $$\lim_{t \to \infty}\frac{\pi_{s t}^{-1}(\alpha s t)}{t} = (\alpha + U) \vee (\sqrt{\alpha}-1)^2 \times s,$$ in the sense of convergence in distribution with respect to the uniform topology on functions on $[0,+ \infty)$
        \item $$\lim_{t \to \infty}\left( \frac{\pi_t^{-1}(1) - t}{\sqrt{t}}, \frac{\pi_t^{-1}(2) - t}{\sqrt{t}}, \dots \frac{\pi_t^{-1}(l) - t}{\sqrt{t}} \right) \overset{(d)}{\rightarrow}  (\lambda_1^1, \lambda_2^2, \dots \lambda_l^l),$$
    \end{enumerate}
    where $U$ is a uniform random variable on $[-1,1]$, and $(\lambda_1^1, \lambda_2^2, \dots \lambda_l^l)$ is the GUE corners process (see Section \ref{sec:Pre}).
\end{theorem}

Next, our objective is to investigate the properties of $\pi_{t}^{-1}(1)$ as a stochastic process. In this direction, our first results describe the number of successful jumps at position $1$ within a given time and the approximate height of each jump when the time is large.

\begin{theorem}[Theorem \ref{thm5.19}]\label{thmt}
     Let $J_t$ be the number of successful jumps occurring at position 1 during time $[0, t]$. Then, we have
     $$\lim_{t \to \infty}\mathbb{E}\left[\frac{J_t}{\sqrt{t}}\right] = \frac{2}{\sqrt{\pi}}. $$
     Let $A_t$ be the event that there is a swap at position $1$ at time $t$ then conditioning on $A_t$(formally, this event has probability 0, we provide details in Section \ref{sec:Main}), the height of that swap $H_t$ is of order $\sqrt{t}$. Moreover, we have:
     $$\lim_{t \to \infty}\mathbb{E}\left[\frac{H_t}{\sqrt{t}} | A_t\right] = \sqrt{\pi}.$$
\end{theorem}

Theorem \ref{thmt} deals only with a single time $t$. Our next goal is to quantify the joint distribution at two different times $t_1$ and $t_2$. Although at a fixed time $t$, $\pi_{t}^{-1}(1)$ can be precisely characterized through the inversion symmetry, which will be developed in Section \ref{sec:Tra} and Section \ref{sec:Flu}, its joint distribution at two different times is significantly more complicated. Although the marginal distribution at any given time $t$ is Poisson, the process is very far from being a Poisson process. 

The analysis of $\pi_{t}^{-1}$ relies heavily on the symmetry properties of the swap process. These properties have been established in various studies, such as \cite{angel2007random}, \cite{angel2009oriented}, \cite{borodin2021color}, but all of them only hold for single time $t$. From the perspective of Hecke algebra \cite{bufetov2020interacting}, the joint distribution loses a concise form when multiple times are considered, making it challenging to handle technically. 

In this paper, we only present a weaker statement than the full description of two-point correlation function of $\pi_{t}^{-1}$. Our main result shows that the next jump on the scale $\sqrt{t}$, given a large time $t$, is asymptotically a folded normal distribution. 

\begin{theorem}[Theorem \ref{thm5.2}]\label{thmty}
    Define $\tau(t) = \inf \{s: \pi^{-1}_t(1) < \pi^{-1}_{t+s}(1)\}$, let $\mathcal{G}$ be the absolute value of a Gaussian random variable with mean $0$ and variance $2$. Then, we have:
    $$
    \lim_{t \to \infty}\frac{\tau(t)}{\sqrt{t}} = \mathcal{G}, 
    $$
    in distribution. 
\end{theorem}

We can phrase the previous theorem in a more explicit way: Consider the probability before the limit,  $\mathbb{P}(\frac{\tau(t)}{\sqrt{t}} < c) = \mathbb{P}(\tau(t) < c \sqrt{t}) = \mathbb{P}(\pi^{-1}_t(1) < \pi^{-1}_{t+c \sqrt{t}}(1))$. Therefore, it is equivalent to the calculation of the limit of $\mathbb{P}(\pi^{-1}_t(1) < \pi^{-1}_{t+c \sqrt{t}}(1))$, as $t \to \infty$, which is the quantity we consider later in Section \ref{sec:Main}. 

The rest of the paper is organized as follows. In Section \ref{sec:Pre} we establish the setup of our model and review relevant background material. In Section \ref{sec:Tra} we prove the first statement of Theorem \ref{thmtyy}, which encodes the law of large number of the position of a certain particle. We also provide a description of the fluctuation of $\pi_{t}^{-1}$ and $\pi_{t}^{1}$, as we stated in the second part of \ref{thmtyy}, via a coupling method in Section \ref{sec:Flu}. Finally, Section \ref{sec:Main} is devoted to exploring several properties of $\pi_{t}^{-1}$ as a stochastic process and to the proof of Theorem \ref{thmt} and Theorem \ref{thmty}. 

\textbf{Acknowledgments.} The author would like to express his gratitude to his advisor, Alexey Bufetov, for suggesting this topic and for numerous insightful discussions. 

\section{Preliminaries}
\label{sec:Pre}
\subsection{Construction of the process}
In addition to the definition of the finite oriented swap process provided in the introduction, we offer an equivalent formulation. Consider $n$ particles at positions $1 \dots n$ with initial condition $\pi_{0}^{n}(i) = i$ and there are $n-1$ independent Poisson processes on $\mathbb{R}^{+}$ with rate $1$ consecutive integers, when the Poisson process between $i$ and $i+1$ jumps, there is an attempt of swap, when $\pi(i) < \pi(i+1)$, the swap occurs; otherwise it is suppressed. This description is known as the graphical representation.

The graphical representation enables us to rigorously define the oriented swap process on the half line. The existence of this infinite oriented swap process can be established by using Harris construction \cite{harris1972nearest}, \cite{harris1978additive}. Specifically, for any fixed time $t$, the probability that no swap occurs between $k$ and $k+1$ is $\exp(-t) > 0$, thus the sum of the events of no swaps between the intervals is $\infty$. By the Borel-Cantelli lemma, with probability $1$ we can divide the positive half line into finite segments such that no swaps have occurred between them before time $t$ and construct the infinite process as the union of finite swap processes. 

The totally asymmetric simple exclusion process (TASEP) was introduced to mathematics by Spitzer in \cite{spitzer1970interaction} and has been widely studied over the past decades. It is a dynamical configuration on an interval with particles and holes, with a rate of $1$ the particles jump to its right if there is a hole. An initial condition of TASEP is defined by a function $f$ on each site of the interval,  $f(x) = 0$ indicates a hole at time $0$, $f(x) = 1$ indicates a particle at time $0$. Using the graph construction, TASEP can be defined on both finite and infinite intervals. We denote a TASEP on interval $I$ with initial condition $f$ as $\sigma^{f, I}_{t}$.  I $I = \mathbb{Z}$ we omit the $I$ in the notation, if in addition $f= \mathbf{1}_{x \leq k}$, we denote it as $\sigma^{k}_{t}$. Similarly, $\sigma^{f, I}_{t}(l)$ represents the position of particle $l$ at time $t$ in that TASEP. In comparison with the oriented swap process, TASEP has only particles and holes, rather than many different types of particles.  

\subsection{Hecke algebra and Symmetry}
We briefly outline the concepts of Iwahori–Hecke algebra(referred to as the Hecke algebra in the remainder of this paper)
that are relevant to our work. The content of this section is borrowed from \cite[Ch.~7]{humphreys1992reflection}, for detailed definitions and proofs we refer readers to the same source.  

\begin{definition}[Hecke algebra]
    Let $W$ be a Coxeter group with generators $S$ and length function $l$, the Hecke algebra $\mathcal{H}(W)$ is a linear space with basis $T_w$ and the following rule of multiplication, for $s \in S$ and $w \in W$:
    $$
    T_sT_w = \left \{ \begin{array}{cc}
        T_{sw}, & \text{if \quad}  l(sw) = l(w) + 1,  \\
        a_sT_w + b_s T_{sw}, & \text{if \quad} l(sw) = l(w) - 1.
    \end{array}
    \right .
    $$
\end{definition}

\begin{proposition}
    There exists an involution $i: \mathcal{H}(W) \rightarrow \mathcal{H}(W)$, defined as $i(T_{\omega}) = T_{\omega^{-1}}$, which is an anti-homomorphism, i.e.
    $$
    i(T_r T_{r-1} \dots T_2 T_1) = i(T_1) i(T_2) \dots i(T_{r-1}) i(T_r).
    $$
    In particular, if all the $T_i$ are indexed by the generators of the Hecke algebra, we have 
    $$
    i(T_{s_r} \dots T_{s_1}) = T_{s_1}  \dots T_{s_r}.
    $$
\end{proposition}

Many particle systems can be characterized using the Hecke Algebra or random walk on the Hecke Algebra, see \cite{bufetov2020interacting} for more examples. In this paper, we focus on the oriented swap process. Let $W = S_n$ be the permutation group, $S$ the set of $n-1$ nearest neighbor transpositions as the generators, and $l$ the length function, defined as the inversion number of a permutation. We set $a_s = 1$ and $b_s = 0$ and define a map $\mathcal{P}: S_n \rightarrow \mathcal{H}(S_n)$ such that a permutation $\omega$ is mapped to an element $\mathcal{P}(\omega) = T_{\omega}$ in the basis of the Hecke algebra. 

When a swap attempt occurs at position $i$, $\omega$ transitions to $\tau_i \omega$ if and only if $\omega(i) < \omega(i+1)$. On the other hand, $T_{\tau_i} T_\omega = T_{\tau_i \omega}$ if and only if $\omega(i) < \omega(i+1)$. Now, consider a continuous-time random walk on Hecke algebra, i.e. $W_t = \prod_{i=1}^{r} T_i$, where each $T_i$ is chosen uniformly from all the nearest neighbor transposition, and $r = Poi(t)$. Writing this random walk as a linear combination of the Hecke algebra basis defines a random measure on the basis. The probability that a finite swap process is in configuration $\omega$ at time $t$ corresponds to the expectation of coefficient of $T_{\omega}$ in the decomposition of $W_t$. 

\begin{corollary}\cite[Cor.~2.2]{bufetov2020interacting}\label{thm2.0}
    \begin{enumerate}
        \item For a finite oriented swap process at a fixed time $t$, $\pi_{t}^{n}$ and $(\pi_{t}^{n})^{-1}$ have the same law, the inverse is in the sense of the inversion of a permutation.
        \item For a infinite swap process at a fixed time $t$, and its first $k$ particles, 
        $$
        (\pi_{t}(1), \pi_{t}(2), \dots \pi_{t}(k)) \overset{(d)}{=} (\pi_{t}^{-1}(1), \pi_{t}^{-1}(2), \dots \pi_{t}^{-1}(k)).
        $$
    \end{enumerate}
\end{corollary}
\begin{remark}
    The symmetry property for the finite case is a direct corollary of the property of the involution. For the infinite case, we need again the Harris' construction and apply the symmetry property to each segment, see also the proof of \cite[Thm 1.4]{amir2011tasep}. 
\end{remark}

\subsection{Properties of TASEP and the oriented swap process}
Our analysis of the oriented swap process relies on its connection with TASEP. In this section, we present some propositions about TASEP and how they relate to the oriented swap process. Firstly, we define certain operations on configurations of TASEP or the oriented swap process. A configuration $\rho \in \mathbb{Z}^{I}$ specifies the type of particle at each position and the set of all such configurations is denoted as $\mathcal{S}$.

\begin{definition}
    For any $k \in \mathbb{Z}$, and any configuration $\rho$, where $I$ is a finite or infinite interval, we define the projection operator $T_k : \mathcal{S} \rightarrow \mathcal{S}$ as 
    $$
    (T_k \rho)(x) = \mathbf{1}_{ \{ \rho(x) \leq k \} }.
    $$
    The operator projects a configuration of multi species particle system(e.g. the infinite oriented process) to a TASEP configuration on the same segment with only particles and holes.  
\end{definition}

\begin{theorem}\cite[Lem.~3.1]{angel2009oriented}
    Let $\pi_{t}^{n}$ be a finite oriented process with $n$ particles and $\sigma_{t}^{k,[1,n]}$ is a TASEP on the interval with k particles at the left as initial condition, then we have:
    $$T_k \pi_{t}^{n} = \sigma_{t}^{k,[1,n]}.$$
\end{theorem}

As a corollary, the oriented swap process can be completely characterized by a family of coupled TASEPs with different initial value conditions. The original swap process can be reconstructed by comparing $T_k \pi_{t}^{n} $ and $T_{k-1} \pi_{t}^{n} $, where the difference corresponds to the location of particle $k$ in the oriented swap process. 

\begin{definition}
    For any TASEP configuration $\rho$, i.e. $\rho(x) = 0$ if there is a hole at position $x$, $\rho(x) = 1$ if there is a particle at position $x$. The cut-off $R_k : \mathcal{S} \rightarrow \mathcal{S}$ is defined as an operator, which retains only the k rightmost particles of the configuration.
    $$
    (R_k \rho)(x) = \left \{ \begin{array}{cc}
        \rho(x), & \text{if \quad}  \sum_{y > x} \rho(y) < k, \\
        0,& \text{otherwise. \quad} 
    \end{array}
    \right .
    $$
\end{definition}

\begin{definition}\label{def1}
    For two different configurations $\rho$ and $\rho^{\prime}$, we say they are compatible if they differ at exactly one site, that is, there exists a unique $x$, such that $|\rho(x) - \rho^{\prime}(x) | = 1$. We denote the position of this difference as $\Sigma_{\rho, \rho^{\prime}}$. 
\end{definition}

An example of compatibility is the so-called second class particle in TASEP. We consider two TASEP $\sigma_{t}^{0}$ and $\sigma_{t}^{-1}$ driven by the same family of Poisson process, meaning their particles are trying to swap simultaneously. Through coupling, these systems remain compatible as time progresses. Let $X(t) = \Sigma_{\sigma_{t}^{0}, \sigma_{t}^{-1}}$. It was proven in \cite{mountford2005motion} that $\frac{X(t)}{t}$ converges to a uniform random variable almost surely. Furthermore, this result holds in a more general sense. 

\begin{theorem}\cite[Thm.~1]{mountford2005motion}\label{thm2.1}
    Let U be a random variable uniformly distributed on $[-1,1]$, Let $X(t)$ be the location of the second class particle, then we have: 
    $$
    \lim_{t \to \infty }\frac{X(st)}{t} = \lim_{t \to \infty }\frac{\Sigma_{\sigma_{st}^{0}, \sigma_{st}^{-1}}}{t} \to U \times s.
    $$
    in the sense of convergence in distribution with respect to the uniform topology on functions on positive half line.
\end{theorem}

The operators $R_k$ and $R_{k-1}$ preserve the compatibility. 
\begin{theorem}\cite[Lem.~7.1]{angel2009oriented}\label{thm2.2}
    Let $\rho, \rho^{\prime}$ be two compatible configurations, then $R_{k}(\rho),R_{k-1}(\rho^{\prime})$ are also compatible. More precisely, 
    $$\Sigma_{R_{k}(\rho),R_{k-1}(\rho^{\prime})} = \Sigma_{\rho, \rho^{\prime}} \vee d_{\rho}(k),$$ 
    where $d_{\rho}(j)$ denotes the position of the j-th rightmost particle of configuration $\rho$, $a \vee b = a$ if $a > b$, otherwise $a \vee b = b$. 
\end{theorem}

\begin{remark}
    The sketch proof of the above theorem is as follows. Intuitively, if the difference between $\rho$ and $\rho^{\prime}$ lies among the $r-1$ particles on the right, the difference will remain unchanged after applying the cut-off operator. If the $r-1$ particles on the right side of $\rho$ and $\rho^{\prime}$ are identical, then the difference after applying the cut-off operator is exactly the r-th rightmost particle in configuration $\rho$.
\end{remark}

The hydrodynamic limit result of TASEP characterizes its large-time behavior.
\begin{theorem}\cite[Thm.~1]{rost1981non}
    Let $u,v \in \mathbb{R}$. We have
    $$
    \lim_{t \to \infty} \frac{1}{t} \sum_{ut < j < vt} \sigma_{t}^{0}(j) = \int_{u}^{v} h(x) dx,
    $$
    where
    $$
    h(x) = \left \{ \begin{array}{cc}
        1, & \text{if \quad}  x < 1,  \\
        \frac{1-x}{2},& \text{if \quad} -1\leq x \leq 1,\\
        0, & \text{if \quad}  x > 1.
    \end{array}
    \right .
    $$
\end{theorem}

As a corollary of the hydrodynamical limit of TASEP, we can approximate the position of the k-rightmost particle at large times, which aids in establishing the convergence to the limiting trajectories.
\begin{corollary}\cite{angel2009oriented}\label{thm2.4}
    Recall that $d_{\rho}(j)$ denotes the position of the j-th rightmost particle of configuration $\rho$.For any $x > y >0$, the position $yt$-rightmost particle of TASEP $\sigma_{xt}^{yt}$ divided by $t$ converge in probability, 
    $$
    \lim_{t \to \infty} \frac{1}{t} d_{\sigma_{xt}^{yt}}(yt) = y + x -2\sqrt{xy}.
    $$
\end{corollary}

\subsection{GUE corners process}
The GUE corners process appeared first in the setup of Gaussian unitary ensemble(GUE). 

\begin{definition}
    The Gaussian Unitary Ensemble is an ensemble of random Hermitian matrices, in which the upper-triangular entries are i.i.d. with complex standard Gaussian random variables, and on the diagonal, there are i.i.d. real standard Gaussian random variables. 
\end{definition}

\begin{definition}
     Let $H$ be a GUE random matrix with infinitely many rows and columns, $H_l$ be a $l \times l$ minor of $H$ and $\lambda_l^i$ be its $i$-th-largest eigenvalue of $H_l$. The vector $(\lambda_1^1, \lambda_2^1, \dots \lambda_l^1)$ is called the GUE corners process. 
\end{definition}

The GUE corners process also arises in other probabilistic models within the KPZ universality class. Baryshnikov established the following convergence result for the swapping times in TASEP.

\begin{theorem}\cite[Thm.~0.7]{baryshnikov2001gues}
    Consider a TASEP with step initial condition, let $x(l,n)$ be the time that the $l$-rightmost particle swap with position $n$. Then,
    $$\lim_{n \to \infty}\left(\frac{x(1,n) - n}{\sqrt{n}} ,\frac{x(2,n) - n}{\sqrt{n}},\dots \frac{x(1,n) - n}{\sqrt{n}}\right) \overset{(d)}{=} (\lambda_1^1, \lambda_2^1, \dots \lambda_l^1).$$ 
\end{theorem}

As a corollary, the positions of TASEP particles, after scaling, converge to a variant of the GUE corners process.
\begin{corollary}\label{thm2.5}
    For a fixed $l$, the joint distribution of locations of first $l$ particles from the right in a TASEP with step initial condition converges,
    $$\lim_{t \to \infty}\left(\frac{\sigma_{t}^{0}(1) - t}{\sqrt{t}}, \frac{\sigma_{t}^{0}(2) - t}{\sqrt{t}}, \dots \frac{\sigma_{t}^{0}(l) - t}{\sqrt{t}}\right) \overset{(d)}{=} (\lambda_1^1, \lambda_2^2, \dots \lambda_l^l).$$ 
\end{corollary}
\begin{proof}
    For any $l$ fixed, we have $\sigma_{t}^{0}(l) = \inf \{n : x(l,n) \geq t\}$
    $$
    \begin{aligned}
        \mathbb{P}\left(\frac{\sigma_{t}^{0}(l) - t}{\sqrt{t}} \geq c\right) &= \mathbb{P}(\sigma_{t}^{0}(l) \geq c \sqrt{t} + t)\\
        &= \mathbb{P} \left(\inf{n : x(l,n) \geq t} \geq c \sqrt{t} + t\right)\\
        &= \mathbb{P} \left(x(l,n) \leq n - c \sqrt{n}\right)\\
        &= \mathbb{P} \left(\frac{x(l,n) - n}{\sqrt{n}} \leq -c\right).
    \end{aligned}
    $$
    Since $H$ is a GUE matrix, $\left(-H\right)$ also follows the same distribution since Gaussian random variables are symmetric, so their eigenvalues have the same distribution, we deduce $\lambda_l^1 \overset{(d)}{=} - \lambda_l^l$. Hence $\frac{\sigma_{t}^{0}(l) - t}{\sqrt{t}} \overset{(d)}{=} -  \frac{x(l,n) - n}{\sqrt{n}} \overset{(d)}{=} -\lambda_l^1 \overset{(d)}{=} \lambda_l^l$. 
\end{proof}

\section{Limiting Trajectories}
\label{sec:Tra}
Similar to the result regarding the limiting trajectories in Angel, Holroyd, and Romik's paper \cite{angel2009oriented}, we can generalize the convergence result to the infinite oriented swap. The limiting trajectories are also determined by a random initial speed, as observed in the finite case. 

In addition, to see a non-trivial convergence result, we cannot look at the trajectory of a fixed particle $k$, since it will eventually have speed 1 when all stronger particles are already far to its right. Instead, we study the trajectories of a sequence of particles that grow linearly in time, and we denote this ratio as the $\alpha$. 
\begin{theorem}\label{thm3}
    Let $\alpha \in (0,1)$, we have
    $$\lim_{t \to \infty}\frac{\pi_{s t}(\alpha s t)}{t} = (\alpha + U) \vee (\sqrt{\alpha} -1) 2 \times s,$$ in the sense of convergence in distribution with respect to the uniform topology of functions on $[0,+ \infty)$, where $U$ is a uniform distribution on $[-1 ,1]$.
\end{theorem}
\begin{remark}
     By selecting specific values of the parameters $\alpha
     $ and $s$ in this theorem, we recover the known results. When $\alpha = 0$, $s=1$, the right-hand side simplifies to $1$, which corresponds to the case where we examine $\pi_t(1)$ at time $t$. The law of $\pi_t(1)$ follows a Poisson random variable with parameter $t$, since it is the strongest particle in the system. By the law of large number $\lim_{t \to \infty}\frac{Poi(t)}{t} = 1$. 
     
     Similarly, when $\alpha = 1$, $s=1$, the theorem reduces to the statement that $\pi_t(t)$ converges almost surely to a uniform random variable in $[0,2]$, which corresponds to the strong law of large number of the second class particle in TASEP \cite{mountford2005motion}. 
\end{remark}
\begin{proof}
    The proof relies mainly on putting together known results from the previous section. First, we project the infinite oriented swap process to the TASEP with two types of particles on half line, in which the particles with labels less than $\alpha st$ are of the first class, and the particles with labels greater than it are holes, and the particle with label $\alpha st$ itself is the second class particle. In our notation, we have$\sigma_{st}^{\alpha st,\mathbb{Z}^+} = \Sigma_{\sigma_{st}^{\alpha st}, \sigma_{st}^{\alpha st - 1}}$(recall Definition \ref{def1}). Next, to use the results of single species-TASEP, we further transform the problem using the cut-off operator and Theorem \ref{thm2.2}. Letting $t \to \infty$, we obtain
    $$
    \begin{aligned}
        \lim_{t \to \infty}\frac{\pi_{s t}(\alpha s t)}{t} &= \lim_{t \to \infty} \frac{\Sigma_{\sigma_{st}^{\alpha st,\mathbb{Z}^+}, \sigma_{st}^{\alpha st - 1,\mathbb{Z}^+}}}{t}\\
        &= \lim_{t \to \infty} \frac{\Sigma_{\sigma_{st}^{\alpha st}, \sigma_{st}^{\alpha st - 1} \vee d_{\sigma_{st}^{\alpha st}}(\alpha st)}}{t}\\
        &= \lim_{t \to \infty} \frac{\Sigma_{\sigma_{st}^{\alpha st}, \sigma_{st}^{\alpha st - 1}}}{t} \vee \frac{d_{\sigma_{st}^{\alpha st}}(\alpha st)}{t}\\
        &= \lim_{t \to \infty} \left( \frac{\Sigma_{\sigma_{st}^{0}, \sigma_{st}^{- 1}}}{t} + \frac{\alpha st}{t}\right) \vee \frac{d_{\sigma_{st}^{\alpha st}}(\alpha st)}{t}\\
        &= [(\alpha + U) \vee (\sqrt{\alpha}-1)^2] s.
    \end{aligned}
    $$
    The second to last line to the last line is due to Corollary \ref{thm2.4} and Theorem \ref{thm2.1}.
\end{proof}

\section{Fluctuations of $\pi_t$ and $\pi_t^{-1}$}
\label{sec:Flu}
According to the results of the previous section on limiting trajectories, $\frac{\pi_t(1)}{t}$ converges to $1$ in probability, as $t \to \infty$. In this section, we turn our attention to its fluctuations. Thanks to the symmetry property, it is sufficient to analyze the permutation without inversion. 
\begin{theorem}
     For a fixed $l$, the joint distribution  of locations of the first $l$ particles from the left in the oriented swap process on half line after scaling converges to a variant of GUE corners process,
    $$\lim_{t \to \infty}\left(\frac{\pi_t(1) - t}{\sqrt{t}}, \frac{\pi_t(2) - t}{\sqrt{t}}, \dots \frac{\pi_t(l) - t}{\sqrt{t}}\right) \overset{(d)}{=} \left(\lambda_1^1, \lambda_2^2, \dots \lambda_l^l\right).$$
\end{theorem}
\begin{proof}
    By Corollary \ref{thm2.5}, we have 
    $$
    \lim_{t \to \infty}\left(\frac{\sigma_{t}^{0}(1) - t}{\sqrt{t}}, \frac{\sigma_{t}^{0}(2) - t}{\sqrt{t}}, \dots \frac{\sigma_{t}^{0}(l) - t}{\sqrt{t}}\right) \overset{(d)}{=} \left(\lambda_1^1, \lambda_2^2, \dots \lambda_l^l\right).
    $$
    By Slutsky's lemma\footnote{The lemma says, the sum of two sequences of random variables, one of them converges in distribution, the other converges in probability to a constant, then the sum of the two sequences converges to the sum of the limiting distribution.}, it is sufficient to show that for any $l$, $\lim\limits_{t \to \infty} \pi_{t}(l) - \sigma_{t}^{0}(l) = 0$ almost surely. Consider a TASEP $\sigma_{t}^{0}$(with particles initially at positions $0, -1, -2 ,\dots$) and a swap oriented process $\pi_{t}$(i.e. with particles initially at positions $1, 2 ,\dots$). We assume that both processes are driven by the same family of Poisson processes. Specifically, whenever there is an attempted swap in $\pi_{t}$ between $i$ and $i+1$, the particle in the TASEP at position $i$ will also attempt to jump. 
    
    For $l = 1$, $\sigma_{0}^{0}(1) = 0$ and $\pi_{0}(1) = 1$, before the first time they locate at the same position, the distance between those two particles is a continuous time random walk $N_t$ starting from $1$, with rate $1$ the value of $N_t$ increases or decreases. By the recurrence property of random walks, $N_t$ will hit $0$ in finite time, which implies that two particles with label $1$ will be at the same position after finite time, with probability $1$. After that, by coupling construction, the two particles will jump simultaneously, as they are driven by the same randomness. Consequently, $\sigma_{t}^{0}(1) = \pi_{t}(1)$ for all $t$ after their first coincidence. 
    
    For $l \geq 2$, we only consider $l = 2$, since $l \geq 3$ follows by applying the same argument as the $l = 2$ case. We again couple $\pi_{t}$ with $\sigma_{t}^{0}$. Based on the description in the previous paragraph, we may assume that at time $s$, $\sigma_{s}^{0}(1) = \pi_{s}(1)$. From this, there are three possible cases to consider. 
    
    If we already have $\sigma_{s}^{0}(2) = \pi_{s}(2)$  at time $s$, then no further proof is required.

    If at time $s$, $\sigma_{s}^{0}(2) < \pi_{s}(2) < \sigma_{s}^{0}(1) = \pi_{s}(1)$, then the difference of $\pi_{s}(2)$ and $\sigma_{s}^{0}(2)$ is no longer a random walk in continuous time, as in the swap process, the particle $2$ cannot jump over $1$. However, the process $\pi_{t}(2) - \sigma_{t}^{0}(2)$ can be bounded from above by a continuous-time random walk starting from $\pi_{s}(2) - \sigma_{s}^{0}(2)$. Again, by the random walk recurrence property, in finite time we have $\sigma_{s}^{0}(2) = \pi_{s}(2)$. 

    The last case is when $\sigma_{s}^{0}(2) < \sigma_{s}^{0}(1) = \pi_{s}(1) < \pi_{s}(2)$. Again, by the recurrence property of random walk, case 3 will fall into case 2 in a finite time. 
\end{proof}

By Corollary \ref{thm2.0}, we prove that the fluctuation of the inversion of the process at the same time, i.e. the label of the particle at position $1$ at time $t$ follows the same asymptotic behavior. This proposition reveals the Gaussian fluctuation of the $\pi_t^{-1}(1)$, to the best of our knowledge there are no probabilistic-based proofs without using the symmetry. 
\begin{proposition}\label{thm4}
    At a fixed time t, the same convergence result holds for the joint distribution $\pi_t^{-1}$:
    $$\lim_{t \to \infty}\left(\frac{\pi_t^{-1}(1) - t}{\sqrt{t}}, \frac{\pi_t^{-1}(2) - t}{\sqrt{t}}, \dots \frac{\pi_t^{-1}(l) - t}{\sqrt{t}}\right) \overset{(d)}{=} \left(\lambda_1^1, \lambda_2^2, \dots \lambda_l^l\right).$$
\end{proposition}

\section{$\pi_t^{-1}$ as a process}
\label{sec:Main}
Based on the proposition in the previous section, we already establish the marginal distribution of the process $\pi_t^{-1}$. In this section, we focus on its much more challenging properties as a stochastic process. To characterize the nature of this process, we need some definitions. 

\begin{definition}
     Let $H_t = \pi_{t+}^{-1}(1) - \pi_{t-}^{-1}(1)$, where $\pi_{t+}^{-1}(1)$ and $\pi_{t-}^{-1}(1)$ are the right and left limits. 
\end{definition}

\begin{remark}
    The existence of the left and right limits of $\pi_{t+}^{-1}(1)$ is guaranteed by its monotonicity. $H_t$ is almost everywhere $0$, it is only positive when there is a swap at position 1 time $t$, and $H_t$ records the difference of the labels of the swapped particles. 
\end{remark}

We also need two technical lemmas. 
\begin{lemma}\label{thm5.1}
    Let $N_t^{1}$ and $N_t^{2}$ be two independent Poisson processes on $\mathbb{R}^{+}$ with rate $1$, and set $N_{t,k} := k + N_t^{1} - N_t^{2}$, $T_{k} := \inf \{ t :  N_{t,k} = 0\}$. Then, we have
    $$
    \mathbb{P}(T_k > t) = \sum_{i = 1 - k}^{k} \mathbb{P}(N_t^{1} - N_t^{2} = i).
    $$
\end{lemma}
\begin{proof}
    By using the reflection principle, for any $x \geq 1$, we have
    $$
    \mathbb{P}(T_k > t, N_{t,k} = x) = \mathbb{P}(N_{t,k} = x) - \mathbb{P}(N_{t,k} = -x), 
    $$
    $$
    \begin{aligned}
        \mathbb{P}\left(T_k > t\right) & = \sum_{i=1} \mathbb{P}\left(T_k > t, N_{t,k} = i\right) = \sum_{i=1} \mathbb{P}\left(T_k > t, N_{t,k} = i\right) - \mathbb{P}\left(T_k > t, N_{t,k} = -i\right)\\
        & = \sum_{i = 1}^{2k}\mathbb{P}\left(N_{t,k} = i\right) = \sum_{i = 1}^{2k}\mathbb{P}\left(N_t^{1} - N_t^{2} = i - k\right) =\sum_{i = 1 - k}^{k} \mathbb{P}\left(N_t^{1} - N_t^{2} = i\right)
    \end{aligned}
    $$
\end{proof}

\begin{lemma}\label{lem5.3}
    The probability that a continuous-time random walk starting at $1$ does not hit $0$ by time $t$ is of order $\frac{1}{\sqrt{t}}$. In more detail, we have 
    $$\lim_{t \to \infty} \sqrt{t}\mathbb{P}\left(T_{1} > t\right) = \frac{1}{\sqrt{\pi}}.$$
\end{lemma}

The asymptotic properties of a broader class of random variables called the Skellam distribution are well known. The Skellam distribution represents the difference between two independent Poisson random variables with parameter $\mu_1$,$\mu_2$,  by $p(k;\mu_1,\mu_2)$ we denote the probability that the Skellam distribution takes value $k$.  We have $$p\left(k;\mu,\mu\right) \sim \frac{e^{-k^2/4\mu}}{\sqrt{4\pi \mu}}.$$
The result of this lemma can be found in \cite{abramowitz1972handbook}(9.7.1, page 377). Lemma \ref{lem5.3} is a corollary of this result by taking $k = 0, \mu = 1$.

Before we dive into the theorems, we first describe the process $\pi_t^{-1}$. The stochastic process can be visualized as a step function, where the steps become progressively longer and the jumps between steps increasingly larger as time progresses. This occurs because as time increases $\pi_{t}^{-1}\left(1\right)$ will be a large number, there is little chance that the particle at position $1$ and the particle at position $2$ can swap, which explains why $\pi_{t}^{-1}\left(1\right)$ will remain constant over long time periods. The following theorem provides a mathematical expression for this intuitive description. 
\begin{theorem}\label{thm5.19}
        Let $J_t$ be the number of successful jumps at position $1$ during time $[0,t]$, we have
       $$\lim_{t \to \infty}\mathbb{E}\left[\frac{J_t}{\sqrt{t}}\right] = \frac{2}{\sqrt{\pi}}. $$
       Define $A_{t} := \{ \pi_t^{-1} (1) < \pi_{t+}^{-1} (1)\}$, we have
       $$\lim_{t \to \infty}\mathbb{E}\left[\frac{H_t}{\sqrt{t}} | A_t\right] = \sqrt{\pi}.$$
\end{theorem}
\begin{remark}
    In the second part of the theorem, the expectation is conditioned on an event $A_t$ of probability zero. We consider the conditional expectation $\mathbb{E}[H_t|A_{t}]$ as the limit of $\mathbb{E}[H_t|A_{t,\epsilon}]$ when $\epsilon \to \infty$, where $A_{t,\epsilon} := \{\pi_t^{-1} (1) < \pi_{t + \epsilon}^{-1} (1)\}$. In other words, this can be interpreted as conditioning that there is a jump at time $t$(the probability of more than one jump is negligible), the expectation of the label difference of the swapped particles. 
\end{remark}
\begin{proof}
    For any short time $\delta t$ the probability that a swap occurs at position $1$ is $\mathbb{P}\left(\pi_s^{-1}\left(1\right) < \pi_s^{-1}\left(2\right)\right) \times \delta t$. Applying the symmetry property \ref{thm2.0} we have
    $$
    \begin{aligned}
        \mathbb{E}[J_t] =  \int_0^{t} \mathbb{P}\left(\pi_s^{-1}\left(1\right) < \pi_s^{-1}\left(2\right)\right) ds  = \int_0^{t} \mathbb{P}\left(\pi_s\left(1\right) < \pi_s\left(2\right)\right) ds.
    \end{aligned}
    $$
    By Lemma \ref{lem5.3}, we have $\mathbb{P}\left(\pi_s\left(1\right) < \pi_s\left(2\right)\right) = \frac{1}{\sqrt{\pi s}} + o\left(\frac{1}{\sqrt{s}}\right).$ The result $\lim_{t \to \infty}\mathbb{E}[\frac{J_t}{\sqrt{t}}] = 2/\sqrt{\pi}$ is deduced from the following fact: $$\mathbb{E}[J_t] = \int_0^{t} \frac{1}{\sqrt{\pi s}} + o\left(\frac{1}{\sqrt{s}}\right) ds = \frac{2}{\sqrt{\pi}}\sqrt{t} + o\left(\sqrt{t}\right).$$
    
    To compute the expectation of $\mathbb{E}[H_t]$ again we use symmetry to transform it into a problem of random walk, we use $N_{t,k}, N_{t}^{1}, N_{t}^{2}$ as the same notation that we used in Lemma \ref{thm5.1}: 
    $$
    \begin{aligned}
        \mathbb{E}[H_t| A_t] & = \lim_{\epsilon \to 0} \mathbb{E}[H_t| A_{t,\epsilon}] = \frac{\mathbb{E}[N_{t,1}, T_{1} > t] \epsilon t}{\mathbb{P} \left(T_{1} > t\right)\epsilon t} = \frac{\mathbb{E}[N_{t,1}, T_{1} > t]}{\mathbb{P} \left(T_{1} > t\right)} \\
        & = \frac{\sum_{k = 1}^{\infty} k \mathbb{P}\left(N_{t,1} = k, T_{1} > t\right)}{\mathbb{P} \left(T_{1} > t\right)} = \frac{\sum_{k = 1}^{\infty} k[\mathbb{P}\left(N_{t,1} = k\right) - \mathbb{P}\left(N_{t,1} = -k\right)]}{\mathbb{P} \left(T_{1} > t\right)}\\
        & = \frac{1 + \mathbb{P}\left(N_t^{1} - N_t^{2} = 2\right)}{\mathbb{P} \left(T_{1} > t\right)}.\\
    \end{aligned}
    $$
    We arrive at the desired result by taking $t \to \infty$ and applying Lemma \ref{lem5.3}. 
\end{proof}

Now we turn to the distribution of the process $\pi_t^{-1}(1)$ at two different times. Since the symmetry property (Corollary \ref{thm2.0}) applies only for a fixed time $t$, the study of the joint distribution will be more difficult. The following theorem deals with how long the particle at position $1$ at time $t$ will remain the same.  
\begin{theorem}\label{thm5.2}
    \begin{equation}
        \lim_{t \to \infty}\mathbb{P}\left(\pi^{-1}_t\left(1\right) < \pi^{-1}_{t+ c t^{\alpha}}\left(1\right)\right) = \left \{
            \begin{array}{cc}
            0, & \text{if $\alpha \in (0,\frac{1}{2})$,} \\
            1, & \text{if $\alpha > \frac{1}{2}$,} \\
            \int_{-c}^c \frac{1}{\sqrt{4\pi}} \exp{\left(-\frac{x^2}{4\pi}\right)}, & \text{if $\alpha = \frac{1}{2}$.}
        \end{array}
    \right .\nonumber
    \end{equation}
\end{theorem}

To prove this theorem, we first need to transform the desired event into another form. We note that the necessary condition for two particles in position $1$ and $2$ to be exchangeable is $\pi^{-1}_t\left(1\right) < \pi^{-1}_t\left(2\right)$, which leads to the following definition: let $R_t$ be the location where the next particle in $\mathbb{Z}^{+}$ can be swapped with $\pi^{-1}_t\left(1\right)$, that is, $\pi^{-1}_t\left(1\right) > \pi^{-1}_t\left(i\right)$, for any $i < R_t$, but $\pi^{-1}_t\left(1\right) < \pi^{-1}_t\left(R_t\right)$.

\begin{remark}\label{them5.3}
     We observe that even if the particle at $R_t$ jumps to its right, the rule of the swap ensures that the new particle at position $R_t$ remains the nearest larger particle after $\pi^{-1}_t\left(1\right)$. This implies that $R_t$ is a non-increasing random variable, provided that no swap has occurred at position $1$ after time $t$. In addition, if no swap has occurred during time $[t,t+s]$, $ - R_t$, as a random process in $[t,t+s]$ is a Possion process with rate $1$. Thus, given $R_t = k$, the probability of observing a new label at position $1$ within time $\tau$ is given by :$$\mathbb{P}\left(\pi^{-1}_t\left(1\right) < \pi^{-1}_{t + \tau}\left(1\right)\right) = \mathbb{P}\left(Poi\left(\tau\right) \geq k-1\right).$$
\end{remark}

\begin{lemma}
    Let $\tau \in \mathbb{R}^{+}$,
    $$
    \mathbb{P}\left(\pi^{-1}_t\left(1\right) < \pi^{-1}_{t+ \tau}\left(1\right)\right) = \sum_{k=2}^{\infty} [\mathbb{P}\left(Poi\left(\tau\right) = k - 1\right)][\sum_{i = 2-k}^{k - 1} \mathbb{P}\left(N_t^{1} - N_t^{2} = i\right)],
    $$
    where $N_t^{1}, N_t^{2}$ are two independent Poisson processes on $\mathbb{R}^{+}$.
\end{lemma}
\begin{proof}
    We start by rewriting the event $\{R_t = k\}$ as the difference of two events, 
    $$
    \begin{aligned}
        \{R_t = k\} &= \{R_t \leq k\} \cap \{R_t \leq k - 1\}^{c}\\
                &= \{\pi^{-1}_t\left(1\right) < \max_{i \in [2,k]}\left(\pi^{-1}_t\left(i\right)\right)\} \cap \{\pi^{-1}_t\left(1\right) < \max_{i \in [2,k-1]}\left(\pi^{-1}_t\left(i\right)\right)\}^{c}.
    \end{aligned}   
    $$
    Thus,
    $$
    \mathbb{P}\left(R_t = k\right) = \mathbb{P}\left(\pi^{-1}_t\left(1\right) < \max_{i \in [2,k]}\pi^{-1}_t\left(i\right)\right) - \mathbb{P}\left(\pi^{-1}_t\left(1\right) < \max_{i \in [2,k-1]}\pi^{-1}_t\left(i\right)\right).
    $$
    Applying Corollary \ref{thm2.0}, we get the following expression, 
    $$
    \mathbb{P}\left(R_t = k\right) = \mathbb{P}\left(\pi_t\left(1\right) < \max_{i \in [2,k]}\left(\pi_t\left(i\right)\right)\right) - \mathbb{P}\left(\pi_t\left(1\right) < \max_{i \in [2,k-1]}\left(\pi_t\left(i\right)\right)\right).
    $$
    The event $\{ \pi_t\left(1\right) < \max\left(\pi_t\left(2\right), \dots \pi_t\left(k\right)\right)\}$ can be interpreted as a simple random walk in continuous time $N_{t,k-1}$ with initial value $N_{0,k-1} = k-1$, which does not hit $0$ until time $t$. Then by Lemma \ref{thm5.1}, we have
    $$
    \begin{aligned}
        \mathbb{P}\left(R_t = k\right) = &\sum_{i = 2 - k}^{k - 1} \mathbb{P}\left(N_t^{1} - N_t^{2} = i\right) - \sum_{i = 3 - k}^{k - 2} \mathbb{P}\left(N_t^{1} - N_t^{2} = i\right) \\
        &= \mathbb{P}\left(N_t^{1} - N_t^{2} = 2 - k\right) + \mathbb{P}\left(N_t^{1} - N_t^{2} = k - 1\right) \\
        &= \mathbb{P}\left(N_t^{1} - N_t^{2} = k - 2\right) + \mathbb{P}\left(N_t^{1} - N_t^{2} = k - 1\right),
    \end{aligned}    
    $$
    where $N_t^{1}$ and $N_t^{2}$ are two independent Poisson processes with rate $1$. 
    Due to the law of total probability and Remark \ref{them5.3}, we can write explicitly $\mathbb{P}(\pi^{-1}_t(1) < \pi^{-1}_{t+ \tau}(1))$, 
    $$
    \begin{aligned}
        \mathbb{P}(\pi^{-1}_t(1) < \pi^{-1}_{t+ \tau}(1)) &= \sum_{k=2}^{\infty} \mathbb{P}(\pi^{-1}_t(1) < \pi^{-1}_{t+ \tau}(1), R_t = k)  = \sum_{k=2}^{\infty} \mathbb{P}(\pi^{-1}_t(1) < \pi^{-1}_{t+ \tau}(1) | R_t = k) \mathbb{P}(R_t = k)\\
        & = \sum_{k=2}^{\infty} \mathbb{P}(Poi(\tau) \geq k - 1) (\mathbb{P}(N_t^{1} - N_t^{2} = k - 2) + \mathbb{P}(N_t^{1} - N_t^{2} = k - 1))\\
        &= \sum_{k=2}^{\infty} \mathbb{P}(Poi(\tau) = k - 1)\sum_{i = 2-k}^{k - 1} \mathbb{P}(N_t^{1} - N_t^{2} = i).
    \end{aligned}
    $$
   The second to last line to the last line is due to the Abel transformation.
\end{proof}

\begin{proof}[Proof of Theorem \ref{thm5.2}]
    If $\alpha < \frac{1}{2}$, 
    $$
    \begin{aligned}
        \mathbb{P}\left(\pi^{-1}_t\left(1\right) < \pi^{-1}_{t+ C t^{\alpha}}\left(1\right)\right) &= \sum_{k=2}^{\infty} \mathbb{P}\left(Poi\left(c t^{\alpha}\right) = k - 1\right)\left(\sum_{i = 2-k}^{k - 1} \mathbb{P}\left(N_t^{1} - N_t^{2} = i\right)\right)\\
        &\leq \sum_{k=2}^{\infty} \left(2k-2\right)\mathbb{P}\left(Poi\left(c t^{\alpha}\right) = k - 1\right) \mathbb{P}\left(N_t^{1} - N_t^{2} = 0\right)\\
        &= \sum_{k=1}^{\infty} 2k \mathbb{P}\left(Poi\left(c t^{\alpha}\right) = k\right) \mathbb{P}\left(N_t^{1} - N_t^{2} = 0\right)\\
        &= 2\mathbb{E}\left[Poi\left(c t^{\alpha}\right)\right] \mathbb{P}\left(N_t^{1} - N_t^{2} = 0\right) = 2c t^{\alpha} \mathbb{P}\left(N_t^{1} - N_t^{2} = 0\right).
    \end{aligned}
    $$
    However, $\lim_{t \to \infty}2c t^{\alpha} \mathbb{P}(N_t^{1} - N_t^{2} = 0) = 0$, when $\alpha < \frac{1}{2}$, so $\lim_{t \to \infty} \mathbb{P}(\pi^{-1}_t(1) < \pi^{-1}_{t+ c t^{\alpha}}(1)) = 0$.

    If the case of $\alpha = \frac{1}{2}$ can be proved, then the case of $\alpha > \frac{1}{2}$ is clear by taking $c \to \infty$ in $\alpha = \frac{1}{2}$ case. So it is sufficient to prove the critical $\alpha = \frac{1}{2}$ case.  We may assume $c=1$, the prove of general $c$ follows in exactly the same way. 

    According to the central limit theorem, $Poi(\sqrt{t})$ takes the values at order $\sqrt{t}$ and has fluctuations of order $\sqrt[4]{t}$. We first show that the terms when $k$ does not belong to this interval are sufficiently small.

    We consider the sum $$\sum_{k=t^\frac{1}{2} + t^{\alpha}}^{\infty} \mathbb{P}(Poi(\sqrt{t}) = k ) = \mathbb{P}(Poi(\sqrt{t}) \in [t^\frac{1}{2} + t^{\alpha}, \infty)),$$ where $\alpha > \frac{1}{4}$, by central limit theorem, the sum tends to $0$ when $t \to \infty$. Applying the fact that $\sum_{i = 2-k}^{k - 1} \mathbb{P}(N_t^{1} - N_t^{2} = i) \leq 1$, we obtain

    $$
    \lim_{t \to \infty}\sum_{k=t^\frac{1}{2} + t^{\alpha}}^{\infty} \mathbb{P}(Poi(\sqrt{t}) = k ) (\sum_{i = 2-k}^{k - 1} \mathbb{P}(N_t^{1} - N_t^{2} = i) \leq 1) = 0.
    $$

    Similarly, it can be shown that the main value of the sum is concentrated within the order given by the central limit theorem for $n$ and $m$ as well, $n$ at order $\sqrt{t}$ and $m$ at order $t$ and has fluctuations of order $\sqrt{t}$. To simplify the notation, we add $n = 1 - k, -k , k$ in the expression which are negligible when $t \to \infty$. The last identity in the following computation is obtained from the Stirling's approximation, 

    $$
    \begin{aligned}
        \mathbb{P}(\pi^{-1}_t(1) < \pi^{-1}_{t+ t^{\frac{1}{2}}}(1)) &= 2\sum_{k=2}^{\infty}\sum_{n = 0}^{k}\sum_{m=0}^{\infty} \mathbb{P}(Poi(\sqrt{t}) = k)  \mathbb{P}(Poi(t) = m)\mathbb{P}(Poi(t) = m+n) + o(1)\\
        &= 2\sum_{k=2}^{\infty}\sum_{n = 0}^{k}\sum_{m=0}^{\infty} \exp{(-\sqrt{t} - 2t)} t^{\frac{k}{2}}\frac{t^{2m + n}}{k! m! (m+n)!} + o(1)\\
        &= \sum_{k=2}^{\infty}\sum_{n = 0}^{k}\sum_{m=0}^{\infty} \frac{2}{\sqrt{(2 \pi)^{3}}} \frac{t^{2m + n} \exp{(2m + n + k)} \exp{(-\sqrt{t} - 2t)} t^{\frac{k}{2}}}{k^k m^m (m+n)^{m+n} \sqrt{(m+n)mk}} + o(1).
    \end{aligned}
    $$
    According to the Central Limit Theorem stated earlier, we retain only the major contributors, relegating the remainder to $o(1)$. Setting $n =  i$, $k = \sqrt{t} + j$, $m = t + l$, the probability above can be expressed as:
    $$
        \lim_{c \to \infty} 2\sum_{j=-c}^{c}\sum_{i = 0}^{\sqrt{t}}\sum_{l=-c}^{c}  \frac{1}{\sqrt{(2 \pi)^{3}}} \exp(2l + j + i) \frac{1}{(1 + \frac{j}{\sqrt[2]{t}})^{\sqrt{t}+j}(1 + \frac{l}{t})^{t+l}(1 + \frac{l + i}{t})^{t + l + i}}\frac{1}{t^{\frac{1}{2}} t^{\frac{1}{2}} t^{\frac{1}{4}}}+o(1).
    $$
    For any $l = 1,2, \dots c\sqrt{t}$ we have the fact that can be applied to those three terms as denominator. The proof follows from taking the logarithm and $ln\left(1 + x\right) = x + o\left(1\right)$, we obtain 
    
    $$\lim_{t\to \infty}\frac{(1 + \frac{l}{t})^{t + l}}{\exp{(l + \frac{l^2}{t})}} =\lim_{t\to \infty}\exp\left((t+l) \log(1 + \frac{l}{t}) - l - \frac{l^2}{t}\right) = \lim_{t\to \infty}\exp\left(o(1)\right) = 1.$$

    Applying the above identity, we have 
    $$
    \mathbb{P}(\pi^{-1}_t(1) < \pi^{-1}_{t+ t^{\frac{1}{2}}}(1)) = \lim_{c\to \infty }2\sum_{j=-c}^{c}\sum_{i = 0}^{\sqrt{t}}\sum_{l=-c}^{c} \frac{1}{\sqrt{(2 \pi)^{3}}}(\frac{1}{\sqrt[4]{t}} e^{-\frac{j^2}{2}}) (\frac{1}{\sqrt{t}} e^{-\frac{l^2}{2}}) (\frac{1}{\sqrt{t}} e^{-\frac{(l+i)^2}{2}}) + o(1).
    $$
    Let $\mathcal{N}_1, \mathcal{N}_2$ be two independent standard Gaussian random variable, we have
    $$
    \begin{aligned}
        \lim_{t \to \infty} \mathbb{P}(\pi^{-1}_t(1) < \pi^{-1}_{t+ t^{\frac{1}{2}}}(1)) &= \lim_{c \to \infty} 2 (\int_{-c}^{c}\frac{1}{\sqrt{2\pi}}e^{-\frac{z^2}{2}}dz)\int_{0}^{1} dy \int_{-c}^{c} \frac{1}{\sqrt{2\pi}} e^{-\frac{x^2}{2}} \frac{1}{\sqrt{2\pi}} e^{-\frac{(x+y)^2}{2}} dx\\
        &= 2 (\int_{-\infty}^{\infty}\frac{1}{\sqrt{2\pi}}e^{-\frac{z^2}{2}}dz)\int_{0}^{1} dy \int_{-\infty}^{\infty} \frac{1}{\sqrt{2\pi}} e^{-\frac{x^2}{2}} \frac{1}{\sqrt{2\pi}} e^{-\frac{(x+y)^2}{2}} dx\\
        &= \int_{-1}^{1} dy \int_{-\infty}^{\infty} e^{-\frac{x^2}{2}} e^{-\frac{(x+y)^2}{2}} dx = \mathbb{P}(|\mathcal{N}_1(0,1) - \mathcal{N}_2(0,1)| \leq 1)\\
        &=\int_{-1}^1 \frac{1}{\sqrt{4\pi}} \exp{(-\frac{x^2}{4\pi})}dx.
    \end{aligned}
    $$
\end{proof}

\bibliographystyle{plain}
\bibliography{ref}
\end{document}